\documentclass[11pt,letterpaper, twoside]{article}
\usepackage[english]{babel}
\usepackage{amssymb}
\usepackage{amsmath}
\usepackage{amsthm}
\usepackage{graphicx}
\usepackage{fancyhdr}
\newfont{\bssten}{cmssbx10}
\newfont{\bssnine}{cmssbx10 scaled 900}
\newfont{\bssdoz}{cmssbx10 scaled 1200}

\usepackage{latexsym,amssymb,amsthm,amsxtra}
\usepackage{amsmath}
\usepackage{stmaryrd}
\usepackage{mathrsfs}
\usepackage{tikz} 
\usepackage{pgf} 
{\bf}{\it}
\newtheorem{theorem}{Theorem}

\newtheorem{lemma}{Lemma}

\newtheorem{proposition}{Proposition}
\newtheorem{corollary}{Corollary}

\def\/{\, | \,}

\def\ind{{\mathchoice {\rm 1\mskip-4mu l} {\rm 1\mskip-4mu l}
{\rm 1\mskip-4.5mu l} {\rm 1\mskip-5mu l}}}

\def\N{{\mathbb N}}

\def\Z{{\mathbb Z}}

\def\P{{\mathbb P}}

\newcommand{\R}{{\mathbb R}}

\def\esp#1{{\mathbb E}\left[#1\right]}

\newcommand{\pae}[1]{\mbox{$\lfloor \kern-1pt #1 \kern-1pt \rfloor$}}
\newcommand{\paep}[1]{\mbox{$\lceil \kern-1pt #1 \kern-1pt \rceil$}}

\newcommand\pr[1]{{\mathbb P}\left[#1\right]}

\newcommand\om{{\omega}}
\newcommand\gre{\textbf{e}}
\newcommand\maG{{\mathcal G}}
\newcommand\maH{{\mathcal H}}

\newcommand\maB{{\mathcal B}}
\newcommand\maA{{\mathcal A}}

\newcommand\maC{{\mathcal C}}

\begin{document}
\title{On the stability of a class of non-monotonic systems of parallel queues}
\author{Pascal Moyal}
\maketitle
\begin{abstract}
 We investigate, under general stationary ergodic assumptions, the stability of systems of $S$ parallel queues in which any incoming customer joins the queue of the server having the $p+1$-th shortest workload ($p < S$), or a free server if any.  
This change in the allocation policy makes the analysis much more challenging with respect to the classical FCFS model with $S$ servers, as it leads to the non-monotonicity of the underlying stochastic recursion. We provide sufficient conditions of existence of a stationary workload, which indicate a "splitting" of the system in heavy traffic, into a loss system of $p$ servers plus a FCFS system of $S-p$ servers. To prove this result, we show {\em en route} an original sufficient 
condition for existence and uniqueness of a stationary workload for a multiple-server loss system.  
\end{abstract}
{\bf keywords:} Queueing systems, Parallel queues, Stochastic comparison, Ergodic Theory;\\
{\bf subject class (MSC 2000):} Primary  60K25, Secondary 35B35.

\section{Introduction}
\label{sec:intro}
In this paper, we re-visit the classical question of stability for multiple-server queues, by investigating 
systems of parallel queues in which the service discipline does not necessarily amount to First Come, First Served (FCFS). 

Let us consider a system of $S$ servers ($S \ge 1$), to which are associated $S$ distinct lines, in a way that 
any entering customer is assigned upon arrival to a line for good, and will wait in that line until service, without possibility 
of switching queue. Such a system is commonly referred to as {\em parallel queues}, as opposed to a {\em multiple-server queue} having a single line and $S$ servers. 
% In the latter system, depending on the service policy, the identity of the server to which an incoming customer is intended to upon arrival may change during the wait. 

For multiple-server queues, the FCFS service discipline is optimal in various senses, see e.g. \cite{Foss81,Foss89,CerFoss01}. 
It is well known, and easily checked, that a multiple-server queue with $S$ servers amounts to a parallel queue applying the following allocation policy: any incoming customer choses upon arrival to join the queue of the server 
having the least {\em workload}, {\em i.e.} the smallest residual work. This allocation policy is usually called {\em Join the Shortest Workload} (JSW) in the context of parallel queues. To set the big picture, 
to any system of $S$ parallel queues with a fixed allocation policy, corresponds 
a multiple-server queue of $S$ servers with a corresponding service discipline, but the converse is not true : in general, 
a multiple-server queue of service discipline different from FCFS is {\em not} a system of parallel queues, as 
the identity of the server to which an incoming customer is intended to upon arrival, may change during the wait.

For the aforementioned optimality reason, most of the literature on the stability of multiple-server queues has addressed 
FCFS systems, or equivalently JSW systems of parallel queues. In fact, the paradigm of parallel queues appeared as more convenient to address the stability of the system in the strongest possible sense. 
Indeed, under the most general assumption (stationary ergodic arrival process and 
service sequence), the queue is simply represented by a stochastic recursive sequence (SRS), keeping track 
upon arrival times, of the residual workload of all servers, in increasing order 
(we call it hereafter the {\em service profile} of the system). The latter SRS follows the famous 
Kiefer and Wolfowitz equation (\ref{eq:KW}), see \cite{KW55}.   
 
Following this approach, Neveu \cite{Neu83} shows a natural 
stability condition (\ref{eq:condstabJSW}) for parallel queues under the JSW policy.   
This stability result is inherited from the monotonicity of the service profile sequence for the 
coordinate-wise ordering, and a minimal stationary profile is given by Loynes's Theorem \cite{Loynes62}. 
Whenever such solution exists, it is in general not unique, and an explicit maximal solution can be obtained, see 
\cite{Brandt85}. 
 
% Despite its prevalence for analytical reasons, the JSW allocation has several shortcomings, as it entails implicitly the knowledge of the workloads of all servers - an information that is not available to the incoming customers. Further, sorting the workloads in increasing order can be expensive algorithmically. Therefore other allocation policy should be taken into account. 
As alternative allocation policies, 
an abundant literature has investigated the so-called {\em Join the Shortest Queue} (JSQ) policy, 
allocating the incoming customers to the server having the shortest line at this instant, 
see \cite{Whitt86,Koole91,KST99,Koole05}. Let us also quote the several existing works (see \cite{SW03,Moy15}), on the {\em semi-cyclic service} allocations (SCS), for which the last $k$ servers (where $k \le S-1$) to which a customer has been assigned are 'frozen', and cannot welcome the next incoming customer. In both cases, the aforementioned papers mostly compare the performance of the alternative model with respect to JSW, however no stability study is undertaken.   
% The extension of these approaches to the case addressed in this paper is not mathematically pertinent: the state descriptor is of completely different nature, as it keeps track of the number of customers in each queue, an information which is not available in the workload profile. 

\medskip

In this paper, we make a first step toward a more general stability study for parallel queues, by generalizing 
the stability results obtained for JSW, to systems ruled by allocations policies of the following class~: send any incoming customer to a server having an empty queue, if any, or to a server having the $p+1$-th smallest workload, where the integer $p \le S-1$ is fixed. This allocation policy is called hereafter J$_{p+1}$SW. 

This generalization makes sense for practical purposes: as is easily seen, the stability of systems under other policies, such as JSQ or SCS, can be addressed by comparison to a J$_{p+1}$SW model. But it also raises an interesting theoretical issue: as will be shown below, this change of policy complicates dramatically the analysis with respect to a JSW system, 
as the underlying SRS is no longer monotonic. Therefore a direct use of Loynes' Theorem is not possible in the present context. 
In particular, it can be easily checked that, although the stability condition of JSW systems is still necessary, it is no longer sufficient for J$_{p+1}$SW systems, as a counter example demonstrates in Section \ref{subsec:NonSuff}. 
We resort to more sophisticated tools, namely the theory of Renovating Events of Borovkov and Foss (see \cite{Bor84,Foss92}) and stochastic comparison, to achieve a sufficient stability condition which is strictly stronger that (\ref{eq:condstabJSW}), 
see Theorem \ref{thm:stabGp}. As will be elaborated in Section \ref{subsec:stab}, this stability condition can be interpreted 
in terms of asymptotic splitting of the model into a loss system of $p$ servers plus an overflow system of $S-p$ servers. 
To prove this, we show {\em en route} a stability result for loss systems with multiple-servers, completing the 
result of \cite{Fli89}. 
 
\bigskip
 
This paper is organized as follows. After some preliminary in Section \ref{sec:prelim}, we introduce our non-monotonic model 
of $S$ parallel queues in Section \ref{sec:model}, and our main stability results for J$_{p+1}$SW queues and Loss systems 
in sub-section \ref{subsec:stab} (namely Theorems \ref{thm:stabGp} and \ref{thm:stabLoss}). 
 Then, in Section \ref{sec:aux} we introduce and study several auxiliary monotonic SRS which will be useful for 
proving the main results. The proofs of Theorems \ref{thm:stabGp} and \ref{thm:stabLoss} are then given in Section \ref{sec:proofs}, in which they are in fact presented as corollaries of stronger more abstract results, namely Propositions 
\ref{prop:stabGp} and \ref{prop:stabLoss}.

\section{Preliminary}
\label{sec:prelim}

\subsection{Main notation}
\label{subsec:notation}
In what follows, $\R$ denotes the real line and $\R_+$, the subset of non-negative numbers. Denote $\mathbb N$ (respectively $\N^*$, $\mathbb Z$), the subset of non-negative (resp. positive, relative) integers. For any two elements $p$ and $q$ in $\N$, $\llbracket p,q \rrbracket$ denotes the finite family 
$\{p,p+1,...,q\}$. For any $x,y \in \R$, denote $x\wedge y=\min\left(x,y\right)$ and $x \vee y=\max(x,y)$. Let $x^+ = x \vee 0$. 
Let $q \in \N^*$. We denote for all $u, v \in \R^q$ and $\lambda \in \R$, 
\begin{align*}
u&=\left(u(1),u(2),...,u(q)\right);\\ 
\lambda u&=\left(\lambda u(1),...,\lambda u(q)\right);\\
u+v &=\left(u(1)+v(1),...,u(q)+v(q)\right);\\
\mathbf 0_q&=(0,\,...,\,0);\\
\gre_i &=(0,\,...,\,\underbrace{1}_{i},\, ... 0)\mbox{ for all }i\in\llbracket 1,q \rrbracket;\\
u^+ &=\Bigl(u(1)^+,u(2)^+,...,u(q)^+\Bigl);\\
\bar u,&\,\mbox{ the fully ordered version of }u,\mbox{ i.e. }\bar u(1) \le ... \le \bar u(q).
\end{align*} 
Denote then $\overline {(\R_+)^q}$ the subset of fully ordered vectors, furnished with the euclidean norm. 
We equip $\overline {(\R_+)^q}$ with the coordinate-wise ordering "$\prec$", {\em i.e.} 
$u \prec v$ if and only if $u(i) \le v(i)$ for all $i\in \llbracket 1,q \rrbracket.$ 
Finally,  if $p$ and $q \in \N^*$ are such that $p \le q$, for any element $u \in \left(\R_+\right)^q$ we denote 
$u[p]$ the restriction of $u$ to its first $p$ coordinates, i.e. 
\[u[p]=(u(1),...,u(p)).\]

\subsection{Stochastic Recursive Sequences}
\label{subsec:SRS}
The stability results presented below will be stated in ergodic theoretical terms, by investigating the existence of a stationary 
version of a stochastic recursion representing the system under consideration. Let us recall the basics of this framework. 
For more details, the reader is referred to the two classical monographs  
\cite{BranFranLis90} and \cite{BacBre02}. 

\medskip

We say that $\left(\Omega,\mathscr F,\mathbb P,\theta\right)$ is a stationary ergodic quadruple if 
$\left(\Omega,\mathscr F,\mathbb P\right)$ is a probability space and $\theta$ is a bijective operator 
on $\mathscr F$ such that for any $\maA \in \mathscr F$, $\pr{\maA}=\pr{\theta^{-1}\maA}$ and any $\theta$-invariant event 
$\maB$ ({\em i.e.}, such that $\maB=\theta^{-1}\maB$) is either negligible or almost sure. Note that this is then also 
true for any $\theta$-contracting event $\maB$ ({\em i.e.} such that $\pr{\maB \,\Delta \,\theta^{-1}\maB}=0$, where $\Delta$ denotes the symmetrical difference). We denote for any $n\in\N^*$, 
\[\theta^n=\underbrace{\theta\circ\theta\circ...\circ\theta}_{n\mbox{ times }}\,\,\mbox{ and }\,\,\theta^{-n}=\underbrace{\theta^{-1}\circ\theta^{-1}\circ...\circ\theta^{-1}}_{n\mbox{ times. }}\]
Let $E$ be a partially ordered Polish space having a minimal point $\mathbf 0$. 
Let $X$ be a $E$-valued random variable (r.v. for short) and $\Psi$ a random measurable mapping from $E$ onto itself. The stochastic recursive sequence (SRS) $\left\{W_{X,n}\right\}$ of initial value $X$ 
and driven by $\Psi$ is defined by
\[
\left\{\begin{array}{ll}
W_{X,0} &=X;\\
W_{X,n+1} &=\Psi\circ\theta^n\left(W_{X,n}\right),\,n\in\N.
\end{array}\right.\]
It is then routine to check that a time-stationary sequence having, on a reference probability space, the same distribution as $\{X_n\}$ corresponds to a solution $X$ defined on $\Omega$ to the functional equation  
\begin{equation}
\label{eq:recurstat}
X\circ\theta=\Psi\left(X\right),\mbox{ almost surely.} 
\end{equation}
As stated by Loynes's Theorem (\cite{Loynes62}, see also the generalization in section 2.5.2 of \cite{BacBre02}), in the case where $\Psi$ is a.s. nondecreasing and continuous, a solution to (\ref{eq:recurstat}), almost surely (a.s., for short) minimal for the 
considered ordering on $E$, is given by the almost sure limit of the so-called 
 Loynes's sequence $\{\mathbf W_n\}$, defined by  
\[\left\{\begin{array}{ll}
\mathbf W_0&=0;\\
\mathbf W_{n+1}&=\Psi\circ\theta^{-1}\left(\mathbf W_n\circ\theta^{-1}\right),\,n\in\N. 
\end{array}\right.\]  
(hereafter Loynes's sequences are always denoted in bold letters). 

\section{Join one of $S$ parallel queues}
\label{sec:model} 
We consider a queueing system having $S$ servers ($S \in \N^*$) working in parallel, each of them providing service in First in, first out (FIFO). We assume that each server has a dedicated line and that switching between lines is not allowed, in other 
words any incoming customers choses a server upon arrival, and remain in the corresponding line until service. 

\medskip

The input of the system is represented by a marked stationary point process. We 
denote $\mathscr Q=\left(\Omega,\mathcal F, \mathbb P, \theta\right)$, the Palm probability space associated to the arrival process. In particular, $\mathbb P$-a.s., a customer (denoted $C_0$) enters the system at time 0, requesting a service time of duration $\sigma$, measured in an arbitrary time unit. The following customer $C_1$ enters the system at time $\tau$. Then, for any $n \in \Z$, $\sigma\circ\theta^n$ is interpreted as the service time requested by customer $C_n$, and $\tau\circ\theta^n$ represents the time epoch between the arrivals of customers $C_n$ and $C_{n+1}$ (again, see 
\cite{BranFranLis90,BacBre02}). In addition, the shift $\theta$ is assumed $\mathbb P$-stationary and ergodic (so $\mathscr Q$ is a stationary ergodic quadruple and we can work in the settings of Section \ref{subsec:SRS}). We also assume that $\sigma$ and $\tau$ are integrable, and that $\pr{\tau>0}=1$.  

\subsection{The service profile} 
We assume that a global information is available to all entering customers, on the quantity of work to be completed by each of the $S$ servers (termed {\em workload} of the corresponding server). In other words, for any $n$ the system can be represented at the arrival of $C_n$ by the random vector $W_n \in \overline{(\R_+)^S}$, where $W_n(i)$ denotes the workload of server $i$, whenever the servers are sorted in the ascending order of workloads. The random vector $W_n$ is called the {\em service profile} at time $n$.  
The {\em allocation policy} is the rule for attributing a server to the incoming customers. The main allocation policy addressed in the literature in that context is {\em Join the shortest Workload} (JSW), in other words any incoming customer joins the queue of the server having the smallest workload. 
It is a weel-known fact since the work of Kiefer and Wolfowitz \cite{KW55}, that the service profile sequence $\{W_n\}$ 
is a SRS driven by the random map 
\[G:\left\{\begin{array}{ll}
 \overline {(\R_+)^S} &\to \overline {(\R_+)^S}\\
    u &\mapsto \overline{\left[u + \sigma\mathbf e_1 -\tau\mathbf 1\right]^+}.
    \end{array}\right.\]
It is immediate to observe that $G$ is a.s. $\prec$-nondecreasing and continuous, so Loynes's Theorem shows the existence 
of a solution $W$ to 
the equation \begin{equation}\label{eq:KW} W\circ\theta=G\left(W\right)\mbox{ a.s.,}\end{equation} 
moreover $W$ is a.s. finite whenever 
\begin{equation}
\label{eq:condstabJSW}
\esp{\sigma} < S\esp{\tau},
\end{equation}
and is such that $\pr{W(1)=0}>0$ (see also \cite{Neu83}). 

\subsection{Alternative non-monotonic allocations}

Consider now the following allocation policies for any $p \in \llbracket 0,S-1 \rrbracket$: 
any incoming customer is sent: 
\begin{itemize}
\item[(i)] to the queue having the $p+1$-th smallest workload among the $S$ servers if all the servers are busy, or
\item[(ii)] to a free server, if any. 
\end{itemize} 
Once in a given line, it is not possible for any customer to leave the queue for another one. 
Such policy will be denote J$_{p+1}$SW for 'Join the $p+1$-th shortest workload'. Observe that J$_1$SW coincides with JSW. 
% Throughout, we add a superscript "$^p$" to all parameters 
% concerning a system of $S$ servers under J$p$SW. In particular, for $p=1$ the customer allocation corresponds to the well-known Join the Shortest Workload (JSW) policy. For this particular case, we omit the superscript "$^{1}$". 

% We will also be led to consider the following alternative allocation policies: for all $p \in \llbracket 1,S \rrbracket$, in the {\em modified $\sP$-system} the new customer joins the $p$-th smallest workload server, regardless of the fact that some server is available at this time or not. All the parameters of this system will be added a "$\,\tilde{ }\,$" and an exponent "$^p$". Observe that this allocation is somewhat pointless, 
% as it leaves in any case, $p-1$ servers eventually inactive (depending on the initial conditions) and therefore amounts to a JSW system of $S-p+1$ servers.  

Fix $p \in \llbracket 0,S-1$, and represent the system under J$_{p+1}$SW, upon the arrival of a customer $C_n$, by the $\overline {(\R_+)^S}$-valued service profile $W^p_n$, defined as above. In other words, $W^p_n(i)$ denotes 
the workload of the server having the $i$-th smallest workload, upon the arrival of $C_n$. 
We can immediately observe the following generalization of Kiefer and Wolfowitz equation: 
 the service profile sequence $\{W^p_n\}$ is an $\overline{\left(\R_+\right)^S}$-valued SRS 
driven by the random map 
\[G^p:\left\{\begin{array}{ll}
 \overline {(\R_+)^S} &\to \overline {(\R_+)^S}\\
    u &\mapsto G\left(u\right)\ind_{\{u(1)=0\}}+\overline{\left[u + \sigma\mathbf e_{p+1} -\tau\mathbf 1\right]^+}\ind_{\{u(1)>0\}}.
    \end{array}\right.\]
Then, from (\ref{eq:recurstat}) a time-stationary service profile is a $\overline {(\R_+)^S}$-valued r.v. $W^p$, solution to the functional equation 
\begin{equation}
\label{eq:recurstatGp}
W^p\circ\theta=G^p\left(W^p\right),\mbox{ a.s..} 
\end{equation} 

The aim of the present paper is to solve (\ref{eq:recurstatGp}) for any $p < S$ and thereby, to generalize the aforementioned 
stability result to any value of $p$. However, we can observe right away that, 
for any $p \in \llbracket 1,S-1 \rrbracket$ the random map $G^p$ is not 
$\prec$-nondecreasing. Consequently, the resolution of (\ref{eq:recurstatGp}) cannot be addressed by a Loynes-type argument such as 
(\ref{eq:KW}). We will circumvent this difficulty by comparing stochastically the service profile sequence to a monotonic 
SRS, namely the one driven by the mapping $\Phi^p$ defined in Section \ref{sec:aux}.  

\bigskip
By investigating the stability of J$_{p+1}$SW systems, we will be led to address the special (and well-known in stability theory) 
case of a pure loss system. As is easily seen, extrapolating to the case $S=p$ in the definition of a J$_{p+1}$SW model, 
leads to the classical loss system of $p$ servers. 
 It is formally defined as follows: there are no lines, so incoming customers enter service if and only if at least one of the servers is free - and begin service with one of these free servers right away. If no server is available, then the incoming customer is immediately lost. 
For each server, the workload is thus given by the remaining service time of the customer in service if any, 
or 0 else. The system is then represented by the service profile sequence $\{U_n^p\}$, where for all $n$, 
$U_n^p(i)$ denotes the $i$-th workload of a server, ranked in ascending order.  
On $\mathscr Q$, it is then easy to check that the service profile sequence $\{U_n^p\}$ is an SRS having the following driving map: 
\[H^p:\left\{\begin{array}{lll}
\overline{\left(\R_+\right)^p} &\longrightarrow &\overline{\left(\R_+\right)^p}\\
                  u &\longmapsto & \overline{\left[u + \sigma\ind_{\{u(1)=0\}}\gre_1-\tau\mathbf 1\right]^+}.
                  \end{array}\right.\]
In other words, for all $u \in \overline{\left(\R_+\right)^p}$ and all $i\in \llbracket 1,p\rrbracket$ we have a.s. 
\begin{equation}
\label{eq:coordSRSLoss}
\begin{array}{ll}
H^p(u)(i) = \left[\left(u(i) \vee \left(\sigma\ind_{\{u(1) =0\}}\right)\right)\wedge u(i+1)-\tau\right]^+;\\
H^p(u)(p) = \left[u(p) \vee \left(\sigma\ind_{\{u(1) =0\}}\right)-\tau\right]^+.\\
\end{array}
\end{equation} 
The question of existence of a stationary service profile then amounts to a solution to the equation 
\begin{equation}
\label{eq:recurstatLoss}
U\circ\theta=H^p(U),\,\P-\mbox{ a.s..}
\end{equation} 

\subsection{Non-sufficiency} 
\label{subsec:NonSuff}
Clearly, condition (\ref{eq:condstabJSW}) is necessary for the stability of any system of $S$ parallel queues.  To see this, on can generalize the argument of 
Exercise 2.3.1  in \cite{BacBre02} : whenever 
$\esp{\sigma} > S\esp{\tau}$, if there existed a stationary service profile $W$ for the system under consideration we would have that  
\begin{align*}
\left(\sum_{i=1}^S W(i)\right)\circ\theta - \sum_{i=1}^S W(i)
%&=\sum_{i=1}^S V\circ\theta(i) - \sum_{i=1}^S V(i)\\
&=\left[W(1)+\sigma-\tau\right]^+ +\sum_{i=2}^S \left[W(2)-\tau\right]^+ - \sum_{i=1}^S W(i)\\
& \ge W(1)+\sigma-\tau-W(1)+\sum_{i=2}^S \left(W(i)-\tau - W(i)\right)= \sigma - S\tau
\end{align*}
and therefore 
$$\esp{\left(\sum_{i=1}^S W(i)\right)\circ\theta - \sum_{i=1}^S W(i)} \ge \esp{\sigma} - S\esp{\tau} >0,$$
a contradiction to the Ergodic Lemma (Lemma 2.2.1 in \cite{BacBre02}). 

\medskip

However, as is intuitively clear, the natural condition (\ref{eq:condstabJSW}) is not sufficient for the stability of a 
system under J$_{p+1}$SW. We can convince ourselves using a counter-example. 
Construct the following simple stationary ergodic quadruple $\mathscr Q$: we let $\Omega=\{\om_1,\om_2,\om_3\}$, $\mathbb P$ be the uniform probability on $\Omega$ and $\theta$ be the automorphism 
$$\theta:\om_1 \mapsto \om_2 \mapsto \om_3 \mapsto \om_1.$$ 
Define on $\Omega$ the following random variables, 
\[\left\{\begin{array}{c}
\sigma(\om_1)=2.25, \sigma(\om_2)=1.5,\, \sigma(\om_3)=2;\\
\tau(\om_1)=\tau(\om_2)=\tau(\om_3)=1,
\end{array}\right.\]
in a way that $$\esp{\sigma}={5.75 \over 3}< 2\esp{\tau}.$$
\begin{proposition}
For $S=2$, $p=1$ and the r.v. $\sigma,\tau$ above, there is no solution to (\ref{eq:recurstatGp}) on the quadruple $\mathscr Q$. 
\end{proposition}
\begin{proof}
Suppose that the r.v. $V^2=\left(V^2(1),V^2(2)\right)$ solves (\ref{eq:recurstatGp}) on $\mathscr Q$. 
Denote in this proof by $X[\om]$, the value of a r.v. $X$ at sample $\om$. Then, 
\begin{enumerate}
\item If $V^2(1)[\om_1]=0$,
then \begin{itemize}
     \item[(1a)] If $V^2(2)[\om_1]\le 1$, then $V^2[\om_2]=(0,1.25),$ which implies that $V^2[\om_3]=(0.25,0.5)$ and in turn 
                   $V^2[\om_1]=(0,1.25)$, an absurdity. 
     \item[(1b)] If $V^2(2)[\om_1]\in (1,2]$, we have $V^2[\om_2]=\left(V^2(2)[\om_1]-1,1.25\right),$ then 
           $V^2[\om_3]=(0,1.75)$ and $V^2[\om_1]=(0.75,1)$, which is also absurd.
     \item[(1c)] $V^2(2)[\om_1]\in (2,2.25]$ entails that $V^2[\om_2]=\left(V^2(2)[\om_1]-1,1.25\right),$  
           $V^2[\om_3]=\left(V^2(2)[\om_1]-2,1.75\right)$ and then $V^2[\om_1]=(0,2.75)$, another contradiction.
     \item[(1d)] If $V^2(2)[\om_1] > 2.25$, then $V^2[\om_2]=\left(1.25,V^2(2)[\om_1]-1\right),$  
           $V^2[\om_3]=\left(0.25,V^2(2)[\om_1]-0.5\right)$ and then $V^2[\om_1]=(0,V^2(2)[\om_1]+0.5)$, again an absurdity.   
     \end{itemize}      
\item Assume now that $V^2(1)[\om_1]\in (0,1]$. Then, 
     \begin{itemize}
     \item[(2a)] If $V^2(2)[\om_1]\le 0.25$, then we also have that $V^2(1)[\om_1]\le 0.25$ and therefore 
      $V^2[\om_2]=(0,V^2(2)[\om_1]+1.25),$ and in turn $V^2[\om_3]=(V^2(2)[\om_1]+0.25,0.5)$ and  
                   $V^2[\om_1]=(V^2(2)[\om_1]-0.75,1.5)$, a contradiction. 
     \item[(2b)] If now $V^2(2)[\om_1] >  0.25$, then  
      $V^2[\om_2]=(0,V^2(2)[\om_1]+1.25),$ and then $V^2[\om_3]=(0.5,V^2(2)[\om_1]+0.25)$ and  
                   $V^2[\om_1]=(0,V^2(2)[\om_1]+1.25)$, another absurdity. 
     \end{itemize}
\item Suppose now that $V^2(1)[\om_1]\in (1,2]$, and in particular $V^2(2)[\om_1]>1$. 
This implies that $V^2[\om_2]=(V^2(1)[\om_1]-1,V^2(2)[\om_1]+1.25),$ and in turn $V^2[\om_3]=(0,V^2(1)[\om_1]+1.75)$ 
and $V^2[\om_1]=(1,V^2(2)[\om_1]+0.75)$, again absurd.        
\item Finally, assuming that $V^2(1)[\om_1]>1$ entails that  
$$V^2[\om_2]=(V^2(1)[\om_1]-1,V^2(2)[\om_1]+1.25),\,\quad V^2[\om_3]=(V^2(1)[\om_1]-2,V^2(2)[\om_1]+1.75)$$ 
and then $V^2[\om_1]=\left(\left[V^2(1)[\om_1]-3\right]^+,V^2(2)[\om_1]+0.75\right)$, another contradiction. This concludes the proof.
\end{enumerate}  
\end{proof}

\subsection{Main results}
\label{subsec:stab}
Define the following family of random variables,   
\begin{equation}
\label{eq:defZ}
Z_\ell=\left[\sup_{k\ge \ell}\left(\sigma\circ\theta^{-k} -\sum_{i=1}^k \tau\circ\theta^{-i}\right) \right]^+,\,
\ell \in \N^*
\end{equation}
which, from Birkhoff's Theorem, are $\mathbb P$-a.s. finite under the ongoing assumptions 
(see Exercise 2.6.1 in \cite{BacBre02}).  

Fix $p \in \llbracket 1,S \rrbracket$ throughout this section. 
Our main result provides a sufficient condition for the existence of a solution to 
(\ref{eq:recurstatGp}),  
\begin{theorem}
\label{thm:stabGp}
Assume that the input is GI/GI (i.e. the inter-arrival times and service times are mutually independent 
i.i.d. sequences) and that the distribution of $\tau$ has unbounded support. 
If moreover the following conditions hold, 
\begin{align}
&\esp{\sigma}\pr{Z_p>0} <(S-p)\esp{\tau};\label{eq:hypoGp1}\\ 
&\pr{Z_p=0} >0,\label{eq:hypoGp2}
\end{align}
% \begin{align}
% &\esp{\sigma}\pr{Y(1)>0} <(S-p)\esp{\tau} ,\,\mbox{ for any solution $Y$ to (\ref{eq:recurstatPsi}) below}, \label{eq:condstabGp}\\ 
% &\pr{Y^p(1)=0} >0\mbox{ for the minimal solution $Y^p$ to (\ref{eq:recurstatPsi})}.\label{eq:hypomain}
% \end{align}
where $Z_p$ is defined by (\ref{eq:defZ}), then the equation (\ref{eq:recurstatGp}) admits a proper solution $W^p$.
\end{theorem}
Theorem \ref{thm:stabGp} is proven in section \ref{subsec:proofGp}, where we in fact show a stronger result, Proposition 
\ref{prop:stabGp}, which exhibits a weakest, although implicitly defined, sufficient condition of stability for a largest class of 
models. 

\bigskip

Let us now turn to the special case of a loss queue. The stability study of this genuine example of non-monotonic 
system, for which it is well known that the uniqueness and even the existence of a solution are not true in general, has a long history in the literature: see {\em e.g.} the early papers \cite{Fli83,Lis82,Neu83} on single server queues, and the  
generalization to a multiple-server system \cite{Fli89}, where an extension of the original probability space is constructed, 
on which a solution to (\ref{eq:recurstatLoss}) always exists. Hereafter, we complete the result in \cite{Fli89} by providing a sufficient condition of existence and uniqueness on the original space, using a generalization of the coupling method of Section 2.6.2 in \cite{BacBre02} to multiple-server queues. We obtain the following,   

\begin{theorem}
\label{thm:stabLoss}
If the input is GI/GI, the distribution of $\tau$ has unbounded support and condition (\ref{eq:hypoGp2}) holds true,  
then the equation (\ref{eq:recurstatLoss}) admits a unique finite solution $U^p$, that is such that  
\begin{equation}
\label{eq:compareUZ}
U^p \prec \left(Z_p,Z_{p-1},...,Z_1\right)\,\mbox{ a.s..}
\end{equation}
\end{theorem}

In section \ref{subsec:proofLoss}, Theorem \ref{thm:stabLoss} is shown to be a corollary of a stronger result, 
Proposition \ref{prop:stabLoss}. 

\bigskip

After this detour by loss systems with multiple-servers, we now have a clearer interpretation of Theorem \ref{thm:stabGp}: 
under (\ref{eq:hypoGp2}), $U^p(1)$ is the unique stationary least workload of a loss system of $p$ servers, so 
$\pr{U^p(1)>0}$ can be interpreted as the loss probability for such system (denote $\P_{\mbox{\tiny{loss}}}$, the latter). 
Therefore, from (\ref{eq:compareUZ}), condition (\ref{eq:hypoGp1}) entails that 
\begin{equation}
\label{eq:interpret}
{\P_{\mbox{\tiny{loss}}} \over \esp{\tau}} < {S-p \over \esp{\sigma}}.
\end{equation}
In other words, the sufficient condition (\ref{eq:hypoGp1}) guarantees that the intensity of the overflow process of a 
loss system of $p$ servers does not exceed the service rate of a JSW system of $S-p$ servers. This suggests that, 
as it approaches saturation, the system tends to behave similarly as a two-parts system: the first one has $p$ servers and no waiting line;  the second one has $S-p$ servers having each its own line and works in FIFO. The customers try to enter the first system and, if the latter is full ({\em i.e.} all $p$ servers are busy), they are re-directed to the second one. 
We do not attempt to prove this "assymptotic splitting" of the system in heavy traffic under the ongoing stationary ergodic 
assumptions. Notice that in this general context, it seems that the bound in (\ref{eq:hypoGp1}) is somewhat optimal, as it relies on the 
"nearest" monotonic approximation of the system. 
However, we have reasons to believe that (\ref{eq:interpret}) is necessary and sufficient for stability under more restrictive, such as Markov, assumptions. The proof of this conjecture would require the use of completely different mathematical tools, and is left for future research.

\section{Three Stochastic recursions}
\label{sec:aux}
In this section we introduce and study three abstract auxiliary stochastic recursions that will be useful for addressing the stability of 
J$_{p+1}$SW systems and loss systems, using stochastic comparison. Throughout this section, we work on the Palm 
space $\mathscr Q$ of the arrival process, under the ongoing assumptions on the random variables $\sigma$ and $\tau$. 

\bigskip

Fix two integers $p \ge 1$, $S\ge 2$ such that $p \le S-1$. Define on $\mathscr Q$ the following random maps,  
\[\Gamma^p:\left\{\begin{array}{lll}
\overline{\left(\R_+\right)^p} &\longrightarrow &\overline{\left(\R_+\right)^p}\\
                  u &\longmapsto &v \,\,\mbox{ such that }\\
                  & &v(j)=\left[u(j+1)-\tau\right]^+,\,\,j\in \llbracket 1,p-1 \rrbracket;\\
                  & &v(p)=\biggl[\Bigl(u(p)\vee \sigma\Bigl) - \tau\biggl]^+,\\
                           \end{array}
                  \right.\] 
\[\Psi^p:\left\{\begin{array}{lll}
\overline{\left(\R_+\right)^p} &\longrightarrow &\overline{\left(\R_+\right)^p}\\
                  u &\longmapsto &v \,\,\mbox{ such that }\\
                & &v(j)=\biggl[\Bigl( u(j)\vee \sigma\Bigl) \wedge\, u(j+1)-\tau\biggl]^+,\,\,j\in \llbracket 1,p-1 \rrbracket;\\
                & &v(p)=\biggl[\Bigl( u(p)\vee \sigma\Bigl) -\tau\biggl]^+,
                  \end{array}
                  \right.\]
%and for all $p \in \llbracket 1,S-1 \rrbracket$, 
\[\Phi^{p}:\left\{\begin{array}{ll}
\overline{\left(\R_+\right)^S} &\longrightarrow \left(\R_+\right)^S\\
                  u &\longmapsto v \,\,\mbox{ such that }\\
                 & v(j)=\biggl[\Bigl( u(j)\vee \sigma\Bigl) \wedge\, u(j+1)-\tau\biggl]^+,\,\,j\in \llbracket 1,p \rrbracket;\\
                 & v(j)=\biggl[\Bigl( u(j)\vee \left(\sigma+u(p+1)\ind_{\{u(1)>0\}}\right)\Bigl) \wedge\, u(j+1)-\tau\biggl]^+,\,j \in \llbracket p+1,S-1 \rrbracket;\\
                 & v(S)=\biggl[\Bigl( u(S)\vee \left(\sigma+u(p+1)\ind_{\{u(1)>0\}}\right)\Bigl)-\tau\biggl]^+.
                  \end{array}
                  \right.\] 
                  
\bigskip

We first have the following elementary result, 
\begin{lemma}
\label{lemma:stabGamma}
% If the condition 
% \begin{equation}
% \label{eq:condstabGamma}
% \pr{\sup_{k\ge 1}\left(\sigma\circ\theta^{-k} -\sum_{i=1}^k \tau\circ\theta^{-i}\right)\le 0}>0,
% \end{equation}
% hols,
The functional equation
\begin{equation}
\label{eq:recurstatGamma}
Z\circ\theta=\Gamma^p(Z),\,\mbox{ a.s.}
\end{equation}
admits the only $\overline{\left(\R_+\right)^p}$-valued solution 
\begin{equation}
\label{eq:solstatGamma}
Z^p:=\left(Z_p,Z_{p-1},...,Z_1\right),
\end{equation}
where the $Z_i$'s are defined by (\ref{eq:defZ}). 
\end{lemma}

\begin{proof}
The random map $\Gamma^p$ is a.s. $\prec$-non decreasing and continuous, hence 
a straightforward application of Loynes's Theorem establishes the existence 
of a minimal $\overline{\left(\R_+\right)^p}$-valued solution $Z^p$ to (\ref{eq:recurstatGamma}), 
thereby satisfying 
\[\left\{\begin{array}{ll}
Z^p(p)\circ\theta &=\left[Z^p(p)\vee \sigma - \tau\right]^+;\\
Z^p(j)\circ\theta &=\left[Z^p(j+1)-\tau\right]^+,\,\,j\in \llbracket 1,p-1 \rrbracket.
\end{array}
\right.\]
But we know from Lemma 1 of \cite{Moy09} that the only solution of the recursion  
\[Z\circ\theta=\left[Z\vee\sigma-\tau\right]^+\] 
is precisely $Z_1,$ so we have $Z^p(p)=Z_1$ a.s..  
This implies that   
% \[Z^{p}(S-1)\circ\theta=[Z_1 - \tau]^+,\] or in other words 
\[Z^p(p-1)=\left[Z_1\circ\theta^{-1}-\tau\circ\theta^{-1}\right]^+=Z_2\,\mbox{ a.s. }\]
and in turn, by an immediate induction, that $Z^p(p-j)=Z_{j+1}$ a.s. for all 
$j\in \llbracket 1,p-1\rrbracket$. 
% The almost sure finiteness of $Z_1$ and thereby, of 
% all coordinates of $Z^p$, is a straightforward consequence of Birkhoff's Theorem, and the fact that 
% $\tau$ is a.s. positive. 
\end{proof}

Consequently, 
\begin{corollary}
\label{cor:stabPsi}
The recursive equation
\begin{equation}
\label{eq:recurstatPsi}
Y\circ\theta=\Psi^p(Y),\,\mbox{ a.s.}
\end{equation}
admits at least one $\overline{\left(\R_+\right)^p}$-valued solution.  
Furthermore, any such solution $Y$ is such that 
\begin{align}
Y(p) &=Z_1 \mbox{ a.s.};\nonumber\\
    Y&\prec Z^p\, \mbox{ a.s.},\label{eq:compareYZ}
\end{align}
for $Z^p$ defined by (\ref{eq:solstatGamma}), and the solution is unique if it holds that 
\begin{equation}
\label{eq:condstabGamma}
\pr{Z_1 = 0}>0.
\end{equation}
\end{corollary}

\begin{proof}
As is easily checked, the map $\Psi^{p}$ is a.s. $\prec$-non-decreasing and continuous, 
so Loynes's Theorem guarantees once again the existence of a solution to (\ref{eq:recurstatPsi}).  
%We denote $Y^{S,\sigma,\tau}$, the $\prec$-minimal solution obtained from Loynes's scheme. 
Fix such a solution $Y$. We have, first, that  
$Y(p)\circ\theta=\left[Y(p)\vee \sigma -\tau\right]^+$ a.s. and therefore $Y(p)=Z_1$ a.s.. 
Consider the event $\maB=\left\{Y \prec Z^p\right\}.$ 
On $\maB$, we have that 
\begin{align*}
Y(p-1)\circ\theta = \left[\left(Y(p-1)\vee \sigma\right)\wedge Z_1 - \tau\right]^+
                  &\le \left[\left(Z_2\vee \sigma\right)\wedge Z_1 - \tau\right]^+\\
                  &\le \left[Z_1-\tau\right]^+=Z_2\circ\theta=Z^p(p-1)\circ\theta.
                  \end{align*}
This implies in turn that
\begin{align*}
Y(p-2)\circ\theta = \left[\left(Y(p-2)\vee \sigma\right)\wedge Z_2 - \tau\right]^+
                  &\le \left[\left(Z_3\vee \sigma\right)\wedge Z_2 - \tau\right]^+\\
                  &\le \left[Z_2-\tau\right]^+=Z_3\circ\theta=Z^p(p-2)\circ\theta 
                  \end{align*}
and then, by induction on $i$, that $Y(p-i)\circ\theta \le  Z_{i+1}\circ\theta$ for all $i\in \llbracket 1,p-1 \rrbracket$. 
In other words, the event $\maB$ is $\theta$-contracting. 

On the other hand, remark that $Z_1 > \sigma$ a.s. would imply by $\theta$-invariance that 
$Z_1\circ\theta>0$ a.s. and therefore that $Z_1\circ\theta=Z_1-\tau$ 
a.s. and hence $\esp{Z_1-Z_1\circ\theta}<0,$ a contradiction to the Ergodic Lemma. Therefore, we have that $\pr{Z_1 \le \sigma}$. But on 
$\{Z_1 \le \sigma\}$, we have in turn that $Y(p-1)\circ\theta=\left[Z_1 -\tau\right]^+$, or in other words 
\[Y(p-1)=\left[Z_1\circ\theta^{-1} -\tau\circ\theta^{-1}\right]^+=Z_2.\]
On and on, we obtain that 
\[\{Z_1 \le \sigma\} \subset \left\{Y = Z^p\right\} \subset \maB.\] 
We have proven that the event on the left-hand side is non-negligible, and that $\maB$ is $\theta$-contracting  
and thereby of probability 0 or 1: therefore $\maB$ is almost sure, or in other words $Y \prec Z^p \mbox{ a.s..}$

\bigskip

To show uniqueness, define the family of random variables
\[\mathscr I=\left\{\overline{(\R_+)^p}-\mbox{ valued random variables }X:\, X\prec Z^p\mbox{ a.s.}\right\}\] 
and the event \[\maA=\{Z_1=0\}.\] 
From the very definition of the maps $\Gamma^p$ and $\Psi^p$ we have for all $u,v$ in $\overline{(\R_+)^p}$ that 
\[u \prec v \Longrightarrow  \Psi^p(u) \prec \Gamma^p(v)\mbox{ a.s..}\] 
Therefore, for any solution $Y$ to (\ref{eq:recurstatPsi}), as $Y \in \mathscr I$ we obtain by an immediate induction that for all $n$, 
\[Y\circ\theta^n=\Psi^p\circ\theta^n(Y) \prec \Gamma^p\circ\theta^n\left(Z^p\right) = Z^p\circ\theta^n.\]
In particular, 
\[\theta^{-n}\maA = \{Z^p=\mathbf 0\} = \{Y=\mathbf 0\},\]
so $\left\{\theta^{-n}\maA\right\}$ is a stationary sequence of renovating events of length 1 (see \cite{Bor84}), for any sequence $\{Y\circ\theta^n\}$ with $Y$ a solution to (\ref{eq:recurstatPsi}). Then, (\ref{eq:condstabGamma}) amounts to $\pr{\maA}>0$, which implies the uniqueness of the solution $Y^p$, using Remark 2.5.3 of \cite{BacBre02}.                
\end{proof}

Let us now focus on the random maps $\Phi^{p}$. First observe that  %for some $p \in \llbracket 1,S-1 \rrbracket.$ We have  

\begin{lemma}
\label{lemma:H}
Let $p\in \llbracket 1,S-1 \rrbracket$. Then,  
\begin{itemize}
\item[(i)] $\overline {(\R_+)^S}$ is almost surely stable by $\Phi^p$,  
\item[(ii)] $\Phi^{p}$ is $\prec$-nondecreasing;
\item[(iii)] For any $u \in \overline {(\R_+)^S}$ and $v\in\overline {(\R_+)^p}$, we have 
\[u[p] \prec v \Longleftrightarrow \Phi^p\left(u\right)[p] \prec \Psi^p(v)\,\mbox{ a.s..}\]
\end{itemize}
\end{lemma}

\begin{proof}
(i) Plainly, if $p\ge 2$, for all $u \in \overline {(\R_+)^S}$ and all $j \in \llbracket 1,p-1 \rrbracket$, 
%\begin{align*}
\[\Phi^{p}(u)(j)=\left[\left(u(j)\vee \sigma\right) \wedge u(j+1)-\tau\right]^+
             \le \left[\left(u(j+1)\vee \sigma\right) \wedge u(j+2)-\tau\right]^+
             =\Phi^{p}(u)(j+1).\]
%\end{align*}
Also, if $p<S$, we have that 
\begin{align*}
\Phi^{p}(u)(p)&=\left[\left(u(p)\vee \sigma\right) \wedge u(p+1)-\tau\right]^+\\
               &\le\left[\left(u(p+1)\vee \left(\sigma+u(p+1)\ind_{\{u(1)>0\}}\right)\right) \wedge u(p+2)-\tau\right]^+
               =\Phi^{p}(u)(p+1),
\end{align*}
whereas for all $j \in \llbracket p,S \rrbracket$,  
\begin{align*}
\Phi^{p}(u)(j)&=\left[\left(u(j)\vee \left(\sigma+u(p+1)\ind_{\{u(1)>0\}}\right)\right) \wedge u(j+1)-\tau\right]^+\\
               &\le\left[\left(u(j+1)\vee \left(\sigma+u(p+1)\ind_{\{u(1)>0\}}\right)\right) \wedge u(j+2)-\tau\right]^+
               =\Phi^{p}(u)(j+1).
\end{align*}

\bigskip
\noindent
(ii) Let $u \prec v$ in $\overline {(\R_+)^S}$. Then for all $j \le p$, 
\begin{equation*}
\Phi^{p}(u)(j)=\left[\left(u(j)\vee \sigma\right) \wedge u(j+1)-\tau\right]^+
             \le \left[\left(v(j)\vee \sigma\right) \wedge v(j+1)-\tau\right]^+
             =\Phi^{p}(v)(j)
\end{equation*}
and for $j \ge p+1$,  
\begin{align*}
\Phi^{p}(u)(j)&=\left[\left(u(j)\vee \left(\sigma+u(p)\ind_{\{u(1)>0\}}\right)\right) \wedge u(j+1)-\tau\right]^+\\
             &\le \left[\left(v(j)\vee \left(\sigma+v(p)\ind_{\{v(1)>0\}}\right)\right) \wedge v(j+1)-\tau\right]^+
             =\Phi^{p}(v)(j).
\end{align*}
\bigskip
\noindent
(iii) Let $u \in \overline {(\R_+)^S}$ and $v\in\overline {(\R_+)^p}$ such that $u[p] \prec v$. Then we have a.s. 
%\begin{align*}
\[\Phi^p(u)(p) =\biggl[\Bigl( u(p)\vee \sigma\Bigl) \wedge\, u(p+1)-\tau\biggl]^+
               \le \biggl[\Bigl( u(p)\vee \sigma\Bigl)-\tau\biggl]^+
               \le \biggl[\Bigl( v(p)\vee \sigma\Bigl)-\tau\biggl]^+=\Psi^p(v)(p),\]
             %\end{align*}
whereas for any $j\in \llbracket 1,p-1 \rrbracket$ we readily have 
\[\Phi^p(u)(p) =\biggl[\Bigl( u(j)\vee \sigma\Bigl) -\tau\biggl]^+
                  \le \biggl[\Bigl( v(j)\vee \sigma\Bigl)-\tau\biggl]^+=\Psi^p(v)(j).\]
\end{proof}

We have the following result, 
\begin{proposition}
\label{pro:stabH}
Let $p < S$. If the following condition holds,
\begin{equation} 
\esp{\sigma\ind_{\{Y(1)>0\}}} <(S-p)\esp{\tau},\,\mbox{ for any solution $Y$ to (\ref{eq:recurstatPsi})}, \label{eq:hypoPhi}
\end{equation}
the functional equation
\begin{equation}
\label{eq:recurstatPhi}
V \circ\theta= \Phi^{p}(V),\, \P-\mbox{a.s.} 
\end{equation}
admits at least one $\overline {(\R_+)^S}$-valued solution $V$.
\end{proposition}

\begin{proof}
From Lemma \ref{lemma:H}, $\overline {(\R_+)^S}$ is a.s. stable by $\Phi^{p}$ and $\Phi^{p}$ is a.s. 
$\prec$-nondecreasing and clearly continuous. So we invoke once again Loynes's Theorem to claim the existence of a 
$\prec$-minimal solution $V^{p}$, 
towards which the Loynes's sequence corresponding to $\Phi^{p}$, which we denote $\{\mathbf V_n^p\}$ , 
tends a.s. coordinate-wise and increasingly. It remains to give the condition under which the r.v. $V^{p}$ is proper. 
Let us first prove that  
\begin{equation}\label{eq:absurd2}
\Bigl[V^{p}(S)=\infty\mbox{ a.s. }\Bigl]\,\, 
\Longleftrightarrow  \Bigl[V^{p}(i)=\infty\mbox{ for all }i\in \llbracket p+1,S \rrbracket \mbox{ a.s. }\Bigl].
\end{equation}
The latter proposition is non trivial only if $p<S-1$. In that case, suppose that 
\begin{equation}
\label{eq:absurd}
V^{p}(S)=\infty\mbox{ a.s. and }V^{p}(p+1)< \infty \mbox{ a.s..}
\end{equation} 
Define for all $n$, 
\[M_n=\left[\mathbf V_n^p(p+1)\ind_{\{\mathbf V_n^p(1)>0\}}-\sigma+\tau\right]^+.\]
The sequence $\{M_n\}$ is a.s. non-decreasing, and its almost sure limit $M$ reads 
\[M= \left[V^{p}(p+1)\ind_{\{V^{p}(1)>0\}}-\sigma+\tau\right]^+.\] 
Thus, we have for all $n$,  
\[
\mathbf V^p_{n+1}(S)\circ\theta= M_n \vee \left[\mathbf V_n^p(S)-\tau\right]^+ =\mathbf V_n^p(S)-\left(\mathbf V_n^p(S)-M_n\right)\wedge \tau
                                                            \le \mathbf V_n^p(S)-\left(\mathbf V_n^p(S)-M\right)\wedge \tau.
                                                            \]
As the sequence $\{\mathbf V_n^p(S)\}$ is a.s. non-decreasing, we therefore have for all $n$ that $$\esp{\left(\mathbf V_n^p(S)-M\right)\wedge \tau}\le 0$$ which, by monotone convergence, implies in turn that 
\[\esp{\left(V^{p}(S)-M\right)\wedge \tau}\le 0.\]
From (\ref{eq:absurd}), we have $V^{p}(S)=\infty$ and $M<\infty$ a.s., which implies that $\esp{\tau}\le 0$, 
an absurdity. Hence as the vector $V^{p}$ is sorted in ascending order, we have proven (\ref{eq:absurd2}).

\bigskip

Now, observe that for all $n \in \N$,  
\[\mathbf V^p_{n+1}(j)\circ\theta=\biggl[\Bigl( \mathbf V_n^p(j)\vee \left(\sigma+\mathbf V_n^p(p+1)\ind_{\{\mathbf V_n^p(1)>0\}}\ind_{\{j>p\}}\right)\Bigl) \wedge\, \mathbf V_n^p(j+1)-\tau\biggl]^+.\]
So taking the a.s. limit yields in particular
\begin{align}
V^{p}(p+1)\circ\theta &=\biggl[\Bigl(\sigma+V^{p}(p+1)\ind_{\{V^{p}(1)>0\}}\Bigl) \wedge\, V^{p}(p+2)-\tau\biggl]^+;\label{eq:VP+1}\\
V^{p}(S)\circ\theta &=\left[V^{p}(S)\vee \left(\sigma+V^{p}(p+1)\ind_{\{V^{p}(1)>0\}}\right)-\tau\right]^+.\label{eq:VS}
\end{align}
Therefore, both events $\{V^{p}(S)=\infty\}$ and $\{V^{p}(j)=\infty;\,j\in\llbracket p+1,S \rrbracket\}$ 
are $\theta$-contracting, and thereby, either negligible or almost sure. 
So in view of (\ref{eq:absurd2}), we are in the following alternative: 
\begin{equation*}
\Bigl[\text{$V^{p}(i) < \infty$ a.s. for all $i\in \llbracket p+1,S \rrbracket$}\Bigl]\,\,\,
\mbox{ or }
%\label{eq:cas2} 
\,\,\,\Bigl[\text{$V^{p}(i) = \infty$ a.s. for all $i\in \llbracket p+1,S \rrbracket$}\Bigl]. 
\end{equation*}

Let for all $n\in\N$, $S_n=\sum_{i=1}^S \mathbf V_n^p(i).$ We have for all $n$, 
\begin{multline*}
S_{n+1}\circ\theta - S_n =\left[\left(\sigma \vee \mathbf V_n^p(1)\right)\wedge \mathbf V_n^p(p+1) - \tau \right]^+-\mathbf V_n^p(1)\\
                \shoveright{+\sum_{j=2;\,j\ne p+1}^S\left(\left[\mathbf V_n^p(j) - \tau\right]^+-\mathbf V_n^p(j)\right)
                            +\left[\mathbf V_n^p(p+1)+\sigma\ind_{\{\mathbf V_n^p(1)>0\}} - \tau\right]^+-\mathbf V_n^p(p+1)}\\
         \shoveleft{=\left[\left(\sigma \vee \mathbf V_n^p(1)\right)\wedge \mathbf V_n^p(p+1) - \tau \right]^+
                     +\sum_{j=2}^S\left[\mathbf V_n^p(j) - \tau\right]^+-\sum_{j=1}^S \mathbf V_n^p(j)}\\
         -\sum_{j=p+2}^{S} \mathbf V_n^p(j)\wedge \tau -\mathbf V_n^p(p+1)\wedge\left(\tau-\sigma\ind_{\{\mathbf V_n^p(1)>0\}}\right) .%\label{eq:bigard2}
\end{multline*}
Suppose that we have $V^{p}(i) = \infty$ a.s. for all $i \in \llbracket p+1,\, S\rrbracket $. 
Then, almost surely from a certain rank we have that 
%\begin{multline*}
\[S_{n+1}\circ\,\theta - S_n 
=\left[\sigma \vee \mathbf V_n^p(1) - \tau \right]^++\sum_{j=2}^S\left[\mathbf V_n^p(j) - \tau\right]^+-\sum_{j=1}^S \mathbf V_n^p(j)
         -(S-p)\tau +\sigma\ind_{\{\mathbf V_n^p(1)>0\}},\]
%\label{eq:peoplelike1}
%\end{multline*} 
so as the sequence $\{S_n\}$ is a.s. non-decreasing, by monotone convergence we obtain that 
\begin{equation}\label{eq:peoplelike1}
0 \le \esp{\left[\sigma \vee V^p(1) - \tau \right]^++\sum_{j=2}^S\left[V^p(j) - \tau\right]^+-\sum_{j=1}^S V^p(j)
         -(S-p)\tau +\sigma\ind_{\{V^p(1)>0\}}}.\end{equation}
         
On another hand, as $V^p \circ\theta=\Phi^p\left(V^p\right)$ a.s., $V^{p}(p+1)=\infty$ a.s. entails that 
\[\left\{\begin{array}{ll}
V^p(j)\circ\theta &=\biggl[\Bigl( V^p(j)\vee \sigma\Bigl) \wedge\, V^p(j+1)-\tau\biggl]^+,\,\,j\in \llbracket 1,p-1 \rrbracket;\\
V^p(p) &=\biggl[\Bigl( V^p(p)\vee \sigma\Bigl) -\tau\biggl]^+,
         \end{array}\right.\]
or in other words 
\[V^p\circ\theta[p] = \Psi^p\left(V^p[p]\right)\mbox{ a.s.}.\]
Hence from Corollary \ref{cor:stabPsi}, we have  
$V^p[p]=Y\mbox{ a.s. for some solution $Y$ of (\ref{eq:recurstatPsi})}$           
which, together with (\ref{eq:peoplelike1}), leads to the conclusion that   
\begin{multline}
\label{eq:peoplelike2}
\biggl[V^{p}(j)=\infty \mbox{ a.s. for all }j\in\llbracket p+1,S \rrbracket \biggl]\\
\shoveleft{\Longrightarrow  \Biggl[\mbox{For some solution $Y$ to (\ref{eq:recurstatPsi}),}}\\
\left.\esp{\sigma\ind_{\{Y(1)>0\}}} 
             \ge (S-p)\esp{\tau}+\esp{\sum_{j=1}^S Y(j)}
                - \esp{\Bigl[\sigma\vee Y(1)-\tau\Bigl]^++\sum_{j=2}^S \left[ Y(j)-\tau\right]^+}\right].
\end{multline}            
\bigskip
Finally, just observe that for any such $Y$,  
\[\sum_{i=1}^p Y(i) \circ\theta= \sum_{i=1}^p \Psi^p\left(Y\right)(i) 
                                 = \Bigl[\sigma\vee Y(1)-\tau\Bigl]^++\sum_{j=2}^S \left[ Y(j)-\tau\right]^+,\]
which implies by the $\theta$-invariance of $\mathbb P$, that 
\[\esp{\sum_{i=1}^p Y(i)} =\esp{\sum_{i=1}^p Y(i) \circ\theta}=\esp{\Bigl[\sigma\vee Y(1)-\tau\Bigl]^++\sum_{j=2}^S \left[ Y(j)-\tau\right]^+}.\]
Plugging this into the right-hand side of (\ref{eq:peoplelike2}), concludes the proof.  
\end{proof}
Let us observe furthermore, 
\begin{proposition}
\label{prop:minsol}
Denote $Y^p$ and $V^p$, the $\prec$-minimal solutions to (\ref{eq:recurstatPsi}) and 
(\ref{eq:recurstatPhi}) respectively, obtained by Loynes' Scheme. Then we have 
\[V^p[p] \prec Y^p\mbox{ a.s..}\]
\end{proposition}
\begin{proof}
We argue by induction. Denote $\{\mathbf Y_n^p\}$, the $\overline{\left(\R_+\right)^p}$- valued Loynes' sequence for the mapping $\Psi^p$ and again, 
$\{\mathbf V_n^p\}$, the $\overline{\left(\R_+\right)^S}$- valued Loynes' sequence for $\Phi^p$. 
 We have $\mathbf 0_S[p] \prec \mathbf 0_p$, and then $\mathbf V^p_n[p] \prec \mathbf Y^p_n$ a.s. for some $n$ entails 
by $\theta$-invariance that 
 $\mathbf V^p_n\circ\theta^{-1}[p] \prec \mathbf V^p_n\circ\theta^{-1}[p]$ a.s. which, in view of assertion (iii) of Lemma \ref{lemma:H}, 
 implies in turn that 
\[\mathbf V^p_{n+1}[p]=\Phi^p\circ\theta^{-1}\left(\mathbf V^p_n\circ\theta^{-1}\right)[p] \prec \Psi^p\circ\theta^{-1}\left(\mathbf Y^p_n\circ\theta^{-1}\right)=\mathbf Y^p_{n+1}\mbox{ a.s..}\] 
Thus we have proven that 
% \begin{equation}
% \label{eq:sickofthat1}
\[\mathbf V^p_n[p] \prec \mathbf Y^p_n \mbox{ a.s. for all }n\in\N,\]
which concludes the proof taking the almost sure coordinatewise limit. 
\end{proof}

\section{Proofs of the main results}
\label{sec:proofs}

\subsection{J$_{p+1}$SW systems of $S$ servers}
\label{subsec:proofGp}
We are now in position to prove Theorem \ref{thm:stabGp}. 
Let us first make the following simple observation, 
\begin{lemma}
\label{lemma:H2}
For any $p\in \llbracket 1,S-1 \rrbracket$, for all $u,v \in \overline {(\R_+)^S}$ such that $u \prec v$, we have 
$$G^p(u) \prec \Phi^{p}(v) \mbox{ a.s..}$$
\end{lemma}

\begin{proof}
Let $u \prec v$. It suffices to observe that for all $j \le p$,
\begin{equation}
\label{eq:compare0}
G^{p}(u)(j) = \left[\left( u(j)\vee \sigma\ind_{\{u(1)=0\}}\right) \wedge u(j+1) -\tau \right]^+
                    \le \left[\left( v(j)\vee \sigma\right) \wedge v(j+1) -\tau \right]^+
             = \Phi^{p}(v)(j),
             \end{equation}
whereas for all $j \in \llbracket p+1,S \rrbracket$, assertion (ii) of Lemma \ref{lemma:H} entails that 
\[G^p(u)(j) = \Phi^{p}(u)(j) \le \Phi^{p}(v)(j).\]
\end{proof}
We therefore have the following result, 
\begin{proposition}
\label{prop:stabGp}
Suppose that (\ref{eq:hypoPhi}) holds true and denote $V^p$ the minimal solution of (\ref{eq:recurstatPhi}). 
If furthermore it holds true that 
\begin{equation}
\pr{\maG}:=\pr{\left\{V^{p}(1)=0\right\}\,\bigcap\,\left(\bigcap_{\ell=2}^{S}\left\{V^{p}(\ell)\le \sum_{i=0}^{\ell-1}\tau\circ\theta^i\right\}\right)}>0,\label{eq:hypoGp1bis}
\end{equation}
then the equation (\ref{eq:recurstatGp}) admits at least one proper solution $W^p$.
\end{proposition}

\begin{proof}
From Proposition \ref{pro:stabH}, a solution to (\ref{eq:recurstatPhi}) exists, and we let again $V^{p}$ be the minimal one. 
The sequence $\left\{V^{p}\circ\theta^n\right\}$ thus corresponds to the stochastic recursion 
of initial value $V^{p}$, driven by the random map $\Phi^p$. 
Consider the following family of $\overline{\left(\R_+\right)^S}$-valued random variables:
$$\mathscr J=\left\{X:\,X\prec V^{p}\mbox{ a.s.}\right\}.$$ 
Fix a r.v. $X \in \mathscr J$ and denote $\left\{W_{X,n}^p\right\}$, the service profile sequence of the J$_{p+1}$SW system, 
starting from an initial profile $X$. It follows from Lemma \ref{lemma:H2} using a simple induction, that 
$W_{X,n}^p \prec V^p\circ\theta^n \mbox{ a.s. for all }n \in \N.$
Therefore, for almost every sample on the event $\theta^{-n}\maG$ (where $\maG$ is defined by (\ref{eq:hypoGp1bis})) we have that 
\[\left\{\begin{array}{ll}
W_{X,n}^p(1)&=0;\\
W_{X,n}^p(\ell)&\le \displaystyle\sum_{i=0}^{\ell-1}\tau\circ\theta^{n+i}  ;\,\ell \in \llbracket 2,S \rrbracket.
\end{array}\right.\]
It is then easy to check that on $\theta^{-n}\maG$, 
\begin{itemize} 
\item in the vector $W_{X,n+1}^p$, the coordinate corresponding to $W_{X,n}^p(2)$ vanishes and the one corresponding to 
$W_{X,n}^p(1)$ becomes\\ $\left[\sigma\circ\theta^n-\tau\circ\theta^n\right]^+;$ 
\item in $W_{X,n+2}^p$, the coordinate corresponding to $W_{X,n}^p(3)$ vanishes, the one corresponding to 
$W_{X,n}^p(2)$ becomes $\left[\sigma\circ\theta^{n+1}-\tau\circ\theta^{n+1}\right]^+$, and the one corresponding to $W_{X,n}^p(1)$ equals\\ $\left[\sigma\circ\theta^n-\tau\circ\theta^n-\tau\circ\theta^{n+1}\right]^+$;
\item[$\vdots$]
\item[$\vdots$]
\medskip
\item $W_{X,n+S-1}^p$ has at least $1$ null coordinate, and its $S-1$ last coordinates all are functions 
of $\left\{\left(\sigma\circ\theta^{n+i},\sigma\circ\theta^{n+i}\right),\,i\in \llbracket 0,S-2 \rrbracket\right\}$. 
\end{itemize}
Consequently, $\left\{\theta^{-n}\maG\right\}$ is a stationary sequence of renovating events of length $S-1$ for any sequence 
$\left\{W_{X,n}^p\right\}$ with $X \in \mathscr J$ (see \cite{Bor84}). So (\ref{eq:hypoGp1bis}) entails the existence 
of a solution to (\ref{eq:recurstatGp}), applying Theorem 4 in \cite{Foss92} (or equivalently, Theorem 1 p.260 in \cite{Bor84} and Corollary 2.5.1 in \cite{BacBre02}). 
\end{proof}

We conclude with the proof of Theorem \ref{thm:stabGp}. 
\begin{proof}[Proof of Theorem \ref{thm:stabGp}]
From the GI/GI assumption, $\sigma$ is independent to 
$\left\{\left(\sigma,\tau\right)\circ\theta^{-i},\,i\in \N^*\right\}$, so we have that 
\[\esp{\sigma}\pr{Z_p>0}= \esp{\sigma\ind_{\{Z_p>0\}}} \ge \esp{\sigma\ind_{\{Y^p(1)>0\}}},\]
where we used (\ref{eq:compareYZ}) in the second inequality. Therefore, 
(\ref{eq:hypoGp1}) entails (\ref{eq:hypoPhi}). On the other hand, (\ref{eq:Y1nul}) holds true and, 
from Proposition \ref{prop:minsol}, implies that $\pr{V^p(1)>0}$. 
This together with the independence assumptions and the unboundedness of $\tau$, clearly entails (\ref{eq:hypoGp1bis}).  
So Proposition \ref{prop:stabGp} applies, which concludes the proof. 
\end{proof}

\subsection{Loss systems of $p$ servers}
\label{subsec:proofLoss}
We now turn to the proof of Theorem \ref{thm:stabLoss}. 
First observe the following pathwise bounds, 
\begin{lemma}
\label{lemma:compareLoss}
For any two elements $u$ and $v$ of $\overline{\left(\R_+\right)^p}$ we have for any $i \in \llbracket 1,p \rrbracket$,  
\[
\biggl[\mbox{For all }j \in \llbracket i,p \rrbracket,\,u(j) \le v(j)\biggl]\,
\Longrightarrow \biggl[\mbox{For all }j \in \llbracket i,p \rrbracket,\,H^p(u)(j) \prec \Psi^{p}(v)(j)\mbox{ a.s.}\biggl].
% u(p) \le v(p) &\Longrightarrow H^p(u)(p) \le \Psi^{p}(v)(p).\label{eq:compareHPsi2}
\]
\end{lemma} 

\begin{proof}
Fix $i \in \llbracket 1,p \rrbracket$ and suppose that $u(j) \le v(j)$ for all $j \in \llbracket i,p \rrbracket$. 
As $\Psi^p$ is $\prec$-non decreasing, we have 
\[
H^{p}(u)(p)  \le \Bigl[u(p)\vee \sigma - \tau \Bigl]^+=\Psi^p(u)(p)\le \Psi^p(v)(p)\mbox{ a.s.}%\label{eq:compare1}
\]
and likewise, for $j \in \llbracket i,p-1 \rrbracket$,  
\[
H^{p}(u)(j)  \le \Bigl[\left(u(j)\vee \sigma\right)\wedge u(j+1) - \tau \Bigl]^+=\Psi^p(u)(j)\le \Psi^p(v)(j)\mbox{ a.s..}%\label{eq:compare1}
\]
\end{proof}

We therefore have the following result, 
\begin{proposition}
\label{prop:stabLoss}
Let $Y^p$ be the minimal solution of (\ref{eq:recurstatPsi}). If we have 
\begin{equation}
\label{eq:hypoLoss}
\pr{\maH}:=\pr{\left\{Y^{p}(1)=0\right\}
\bigcap\,\left(\bigcap_{\ell=2}^{p}\left\{Y^{p}(\ell)\le \sum_{i=0}^{\ell-1}\tau\circ\theta^i\right\}\right)}>0,
\end{equation}
then there exists a unique solution $U^p$ to (\ref{eq:recurstatLoss}), that is such that 
\begin{equation}
\label{eq:compareYU}
U^p \prec Y^p\mbox{ a.s..} 
\end{equation}
\end{proposition} 

\begin{proof}
Denote for any random variable $X$, $\left\{U^p_{X,n}\right\}$ the service profile 
sequence of the loss system with $p$ servers, when setting the initial value as $U^p_{X,0}=X$ a.s.. 
Define the family of $\overline{(\R_+)^p}$-valued random variables 
\[\mathscr K=\left\{X:\,X \prec Y^p\mbox{ a.s.}\right\}.\] 

Fix a r.v. $X\in \mathscr K$. Then Lemma \ref{lemma:compareLoss} implies with an immediate induction that 
for any $n\in\N$, $U^p_{X,n} \prec Y^p\circ\theta^n\,\mbox{ a.s..}$ 
Thus, for almost every sample on the event $\theta^{-n}\maH$ we have that 
\[\left\{\begin{array}{ll}
U^p_{X,n}(1)&=0;\\
U^p_{X,n}(\ell)&\le \displaystyle\sum_{i=0}^{\ell-1}\tau\circ\theta^{n+i}  ;\,\ell \in \llbracket 2,p \rrbracket.
\end{array}\right.\]
We can then show exactly as in the proof of Proposition \ref{prop:stabGp} that $\{\theta^{-n}\maH\}$ (where the event $\maH$ is defined 
by (\ref{eq:hypoLoss})) is a stationary sequence of renovating events of length $p-1$ for any such sequence $\left\{U^p_{X,n}\right\}.$  
In view of (\ref{eq:hypoLoss}), Theorem 4 of \cite{Foss92} guarantees once again the existence of a solution 
$U$ to (\ref{eq:recurstatLoss}). 

\bigskip

We now prove that 
\begin{equation}
\label{eq:compareYU0}
U \prec Y^p \mbox{ a.s. for any solutions $U$ to (\ref{eq:recurstatLoss}).}
\end{equation}
Fix a solution $U$. Let us first prove that 
\begin{equation}
\label{eq:compareYU1}
\pr{\maC}:=\pr{U(p) \le \sigma\ind_{\{U(1)=0\}}}>0.
\end{equation} 
Suppose that $U(p) > \sigma\ind_{\{U(1)=0\}}$ a.s.. In particular, we have $U(p) > 0$ a.s. and in turn by $\theta$-invariance, 
$U(p)\circ\theta > 0$ a.s.. From (\ref{eq:coordSRSLoss}), we thus have 
\[U(p)\circ\theta = \left[U(p) \vee \left(\sigma\ind_{\{U(1) =0\}}\right)-\tau\right]^+ = U(p)-\tau\]
and hence $\esp{U(p)\circ\theta-U(p)}<0,$ a contradiction to the Ergodic lemma. Hence (\ref{eq:compareYU1}). 

\medskip

On another hand, let us define the events 
\[\maB_i:=\Bigl\{U(j) \le Y^p(j)\,\mbox{ for any }j\in \llbracket i,p \rrbracket\Bigl\};\,\,i\in \llbracket 1,p \rrbracket.\]
From Lemma \ref{lemma:compareLoss} we have for any $i$, almost surely on $\maB_i$, 
\[U(j)\circ\theta=H^p(U)(j) \le \Psi^p(Y^p)(j)=Y^p(j)\circ\theta\,\mbox{ for any }j\in \llbracket i,p \rrbracket,\] 
in other words the events $\maB_i$, $i\in \llbracket 1,p \rrbracket$, all are $\theta$-contracting.  

\medskip 

Now, on $\maC$, we have that 
\[U(p)\circ\theta = \left[\sigma\ind_{\{u(1) =0\}}-\tau\right]^+ \le \left[\sigma-\tau\right]^+\le \left[\sigma\vee Y^p(p)-\tau\right]^+ 
=Y^p(p)\circ\theta.\]
This means that $\maC \subset \theta^{-1}\maB_p$ and therefore \[\pr{\maB_p}=\pr{\theta^{-1}\maB_p}\ge \pr{\maC}>0.\] 
Consequently, the $\theta$-contracting event $\maB_p$ is almost sure. 

\medskip

If $p\ge 2$, suppose that $\maB_{i}$ is almost sure for some $i \in \llbracket 2,p \rrbracket$. 
Then, for almost every sample on $\maC \cap \maB_i$ we have that 
\begin{align*}
U(i-1)\circ\theta &= \left[\left(U(i-1) \vee \left(\sigma\ind_{\{U(1) =0\}}\right)\right)\wedge U(i)-\tau\right]^+\\
                  &\le \left[\left(\sigma\ind_{\{U(1) =0\}}\right)\wedge U(i)-\tau\right]^+\\
                   &\le \left[\left(\sigma\ind_{\{U(1) =0\}}\right)\wedge Y^p(i)-\tau\right]^+
                  \le \left[\left(Y^p(i-1)\vee\sigma\ind_{\{U(1) =0\}}\right)\wedge Y^p(i)-\tau\right]^+
                  =Y^p(i-1)\circ\theta.
\end{align*}
Therefore, $\left(\maC \cap \maB_i\right) \subset \theta^{-1}\maB_{i-1}$, which entails that 
\[\pr{\maB_{i-1}}=\pr{\theta^{-1}\maB_{i-1}} \ge \pr{\maC \cap \maB_i}>0,\] 
so $\maB_{i-1}$ is almost sure. We conclude by induction on $i$ that the events 
$\maB_{i}$, $i \in \llbracket 2,p \rrbracket$ are all almost sure, which concludes the proof of (\ref{eq:compareYU0}).

\bigskip 
An immediate consequence of (\ref{eq:compareYU0}) is the uniqueness of the solution $U^p$ to (\ref{eq:recurstatLoss}): 
indeed, (\ref{eq:compareYU0}) readily entails that $\{\theta^{-n}\maH\}$ is a stationary sequence of renovating events 
of length $p-1$ for any sequence $\left\{U\circ\theta^n\right\},$ with $U$ a solution to (\ref{eq:recurstatLoss}). 
Remark 2.5.3 of \cite{BacBre02} then shows the uniqueness of the solution.   
\end{proof}

\begin{proof}[Proof of Theorem \ref{thm:stabLoss}]
From (\ref{eq:compareYZ}) and (\ref{eq:solstatGamma}), we have that $Y^p(1)<Z_p$ a.s. and thus (\ref{eq:hypoGp2}) 
entails that 
\begin{equation}
\label{eq:Y1nul}
\pr{Y^p(1)=0}>0.
\end{equation}  
As in the proof of Theorem \ref{thm:stabGp}, this together with the GI/GI assumption and the unboundedness of $\tau$, entails 
(\ref{eq:hypoLoss}). We therefore can apply Proposition \ref{prop:stabLoss}. 
\end{proof}

%\bibliographystyle{amsalpha}
%\bibliography{bibliographie4}
\providecommand{\bysame}{\leavevmode\hbox to3em{\hrulefill}\thinspace}
\providecommand{\MR}{\relax\ifhmode\unskip\space\fi MR }
% \MRhref is called by the amsart/book/proc definition of \MR.
\providecommand{\MRhref}[2]{%
  \href{http://www.ams.org/mathscinet-getitem?mr=#1}{#2}
}
\providecommand{\href}[2]{#2}

\end{document}